\documentclass[12pt]{amsart}
\usepackage{amsmath}
\usepackage{amssymb}
\usepackage{graphicx}
\usepackage{pifont}
\usepackage{epsfig}
\usepackage{amscd}
\usepackage{latexsym}
\usepackage{amsfonts}

\usepackage{color}
\newcommand{\cb}{\color{black}}

\newtheorem{theorem}{Theorem}

\newtheorem{lemma}{Lemma}[section]

\newenvironment{Eqnarray*}{\arraycolsep 0.14em\begin{eqnarray*}}
{\end{eqnarray*}\hspace*{-1mm}}

\newcommand\be{\begin{equation}}
\newcommand\ee{\end{equation}}
\newcommand\bea{\begin{eqnarray}}
\newcommand\eea{\end{eqnarray}}
\newcommand\beaa{\begin{eqnarray*}}
\newcommand\eeaa{\end{eqnarray*}}

\newcommand\BR{{\mathbb R}}

\newcommand\qh{\quad\hbox}

\def\eps{\varepsilon}

\begin{document}

\title[No touchdown]
{No touchdown at zero points of the permittivity profile for the MEMS problem}

\author{Jong-Shenq Guo}
\address{Department of Mathematics, Tamkang University,
151, Yingzhuan Road, Tamsui, New Taipei City 25137, Taiwan}
\email{jsguo@mail.tku.edu.tw}

\author{Phlippe Souplet}
\address{Universit\'e Paris 13, Sorbonne Paris Cit\'e,
Laboratoire Analyse, G\'eom\'etrie et Applications, CNRS (UMR 7539),
93430 Villetaneuse, France}
\email{souplet@math.univ-paris13.fr}

\thanks{JSG is partially supported by
the National Science Council of Taiwan under the grant NSC 102-2115-M-032-003-MY3.
PhS is grateful for the hospitality of the Math.~Dept. of the Tamkang University, where part of this work was done.}

\date{\today}

\begin{abstract}
We study the quenching behavior for a semilinear heat equation
arising in models of micro-electro mechanical systems (MEMS).
The problem involves a source term with a spatially dependent potential,
given by the dielectric permittivity profile, and quenching corresponds to a touchdown phenomenon.
It is well known that quenching does occur.
We prove that {\bf touchdown cannot occur at zero points of the permittivity profile.}
In particular, we remove the assumption of compactness of the touchdown set,
made in all previous work on the subject {and whose validity is unknown in
most typical cases.}
This answers affirmatively a conjecture made in
{[\it Y. Guo, Z. Pan and M.J. Ward, SIAM J. Appl. Math {\textbf{66} (2005)}, 309--338]}
on the basis of numerical evidence.
The result crucially depends on a new type~I estimate of the quenching rate, that we establish.
In addition we obtain some sufficient conditions for compactness of the touchdown set, without convexity assumption on the domain.
These results may be of some qualitative importance in applications to MEMS optimal design,
especially for devices such as micro-valves.



\end{abstract}

\maketitle

\section{Introduction {and main results}}
\setcounter{equation}{0}

In this paper, we consider the problem:
\bea
&&u_t= \Delta u+f(x)(1-u)^{-p},\quad {x\in\Omega},\; t>0,\label{pdeGen}\\
&&u=0, \quad  x\in\partial\Omega,\; t>0,\label{bcGen}\\
&&u(x,0)=0,\quad  x\in\Omega,\label{icGen}
\eea
 where $\Omega$ is a bounded domain of $\BR^n$ ($n\ge 1$), {of class $C^{2+\nu}$ for some $\nu>0$,}
\be\label{Hypfp}
\hbox{ $p>0$ and $f$ is a nonnegative, H\"older continuous function on~$\overline\Omega$, with $f\not\equiv 0$.}
\ee

A typical case of interest is the following
\bea
&&u_t={\Delta u}+\lambda {|x|^m}(1-u)^{-2},\quad {x\in\Omega}, \; t>0,\label{pde}\\
&&u=0, \quad {x\in\partial\Omega},\quad t>0,\label{bc}\\
&&u(x,0)=0,\quad {x\in\Omega},\label{ic}
\eea
where $m, \lambda$ are positive constants.
This problem arises in the study of modeling the dynamic deflection of an elastic membrane inside a
micro-electro mechanical system~(MEMS).
The full model is
\be\label{mod}
\epsilon u_{tt}+u_t=\Delta u+\frac{\lambda g(x)}{(1-u)^2\Big(1+\alpha\int_{\Omega}\frac{1}{1-u}dx\Big)^2},
\quad x\in \Omega, \; t>0,
\ee
where $\epsilon$ is the ratio of the interaction due to the inertial and damping terms,
$\lambda$ is {proportional} to the applied voltage, $u$ is the deflection of the membrane
{(the natural physical dimension being thus $n=2$).}
The function $g(x)$, {called the permittivity profile}, represents varying dielectric properties of the membrane.
One of physically suggested dielectric profiles is the power-law profile
$g(x)=|x|^m$ with $m>0$.
The integral in (\ref{mod}) arises due to the fact that the device is embedded
in an electrical circuit with a capacitor of fixed capacitance.
The parameter $\alpha$ denotes the ratio of this fixed capacitance to a reference capacitance of the device.
{As for the initial condition (\ref{icGen}), it means that the membrane has initially no deflection,
the voltage being switched on at $t=0$.}
For the details of background and derivation of this model, we refer the reader to \cite{PT01, JAP-DHB02, FMPS07}.

{The case when $\epsilon=0$ is studied in \cite{GHW08,GK,G-s} for $\alpha>0$ and $f$ a constant.
We shall here concentrate on the case when $\epsilon=\alpha=0$ (so that there is no capacitor in the circuit)
and $f$ is nonconstant.}
It  has been studied extensively for past years, see, e.g., the works \cite{G1, YG-ZP-MJW06, GG08, G08, G08A,  KMS08,YZ10}.
For the study of stationary solutions, we refer to \cite{L89, G97, EGG07,GG07,E08, EG08, GW08, KMS08,YZ10}.

\medskip

By the standard parabolic theory, there exists a unique classical solution of (\ref{pdeGen})-(\ref{icGen}) in a short time interval.
Also, by the strong maximum principle, we have $u>0$ in {$\Omega$ for $t>0$.}
Moreover, the solution $u$ of (\ref{pdeGen})-(\ref{icGen}) can be continued as long as $\max_{x\in\overline\Omega}u(x,t)<1$.
We shall let $[0,T)$ be the maximal existence time interval of $u$, where $T\le\infty$.
If $T<\infty$, then quenching occurs in finite time, i.e.
$$\limsup_{t\to T^-}\ \{\max_{x\in\overline\Omega}u(x,t)\}=1.$$
It is well known {(see, e.g., \cite{EGG,KMS08} and the references therein)} 
that the solution of (\ref{pdeGen})-(\ref{icGen}) quenches in finite time when $\lambda$ is sufficiently large.
A point $x=x_0$ is a {\it quenching point} if there exists a sequence $\{(x_n,t_n)\}$
in $\overline\Omega\times(0,T)$ such that
$$\hbox{$x_n\to x_0$, $t_n\uparrow T$ and $u(x_n,t_n)\to 1$ as $n\to\infty$.}$$
The set of all quenching points is called the quenching set, denoted by $\mathcal{Q}$.
In the context of MEMS, quenching corresponds to a touchdown phenomenon.

\medskip

{Note that in the typical case of \eqref{pde}, there is no source at $x=0$ due to the spatially dependent coefficient $|x|^m$.}
A long-standing open problem, {even in one space dimension,} is to determine whether or not $x=0$ is a quenching point.
More generally, for problem (\ref{pdeGen})-(\ref{icGen}), the question is whether a point $x_0$ such that $f(x_0)=0$
can be a quenching point.
In \cite{GG08,G08}, under the assumption that the quenching set is a compact subset of~${\Omega}$,
it is shown that {$x_0$ is not a quenching point if $f(x_0)=0$.}
On the other hand, {\cb the compactness assumption was proved in \cite{G08} by adapting a moving plane argument from \cite{FM,G1}}
when $f$ is constant or, more generally, when $f$ is nonincreasing as one approaches the boundary.
However, for {the typical} problem (\ref{pde})-(\ref{ic}) it is unknown whether the quenching set is compact.
Actually, supported by numerical evidence provided in \cite{YG-ZP-MJW06}, the following conjecture was made
(see~\cite{GG08,G08, EGG}):

\medskip
{\bf Conjecture.} {\it The point $x=0$ is not a quenching point for problem (\ref{pde})-(\ref{ic}).}
\medskip

{\bf In the present paper, we give an affirmative answer to this conjecture,
as well as for the case of general $f$,  in any space dimension.}
Our main result is the following.

\begin{theorem}  \label{MainThm}
Assume (\ref{Hypfp}) and let the solution $u$ of problem (\ref{pdeGen})-(\ref{icGen})
be such that $T<\infty$.
If $x_0\in  \Omega$ is such that $f(x_0)=0$, then $x_0$ is not a quenching point.
\end{theorem}

In particular, as a special case, we have that $0$ is not a quenching point for problem (\ref{pde})-(\ref{ic}).
Actually, we have been able to answer this question {\it independently} of the compactness issue of quenching set.
In fact, as a key-step -- of independent interest -- to the proof of Theorem~1,
we prove the following estimate, which in particular guarantees that
the quenching rate is of {\it type I} on any compact subset of $\Omega$.
{In what follows, we denote
$$\delta(x)={\rm dist}(x, \partial\Omega),\quad x\in\overline\Omega,$$
the function distance to the boundary.}

\begin{theorem} \label{MainProp}
Assume (\ref{Hypfp}) and let the solution $u$ of problem (\ref{pdeGen})-(\ref{icGen})
be such that $T<\infty$.
Then there exists a constant $\gamma>0$ {\cb (independent of $x,t$)} such that
\be\label{rate}
1-u(x,t)\ge \gamma\,  \delta(x)\,(T-t)^{1/(p+1)},\quad  x\in \Omega,\ 0<t<T.
\ee
\end{theorem}

{Theorem~\ref{MainProp} will be proved via a nontrivial modification of the Friedman-McLeod method (\cite{FM}, see also \cite{G1}).
Once Theorem~\ref{MainProp} is proved, Theorem~\ref{MainThm} will be deduced by constructing a suitable local supersolution.}

\medskip
The compactness of the quenching set remains an open question.
In particular, we do not know if Theorem~\ref{MainThm} remains true if $f(x_0)=0$ with  $x_0\in\partial\Omega$
(in other words, can a zero {\it boundary} point of the permittivity profile
 be a quenching point ?).
As mentioned above, this cannot occur if we assume in addition that $f$
is nonincreasing as one approaches the boundary.
{Actually, as a consequence of Theorem~\ref{MainProp} and of suitable comparison arguments, we
have been able to obtain two further criteria for the quenching set to be compact.
We note that, unlike in the aforementioned criterion, we do not require any convexity of the domain~$\Omega$.}

\begin{theorem} \label{BoundaryResult}
Assume (\ref{Hypfp}) and let the solution $u$ of problem (\ref{pdeGen})-(\ref{icGen})
be such that $T<\infty$. Assume either
\be\label{HypBdry1}
0<p<1,
\ee
or
\be\label{HypBdry}
f(x)=o\bigl(\delta^{p+1}(x)\bigr)\qh{ as } \delta(x)\to 0.
\ee
Then quenching does not occur near the boundary, i.e. $\mathcal{Q}\subset\Omega$.
\end{theorem}

{Going back to MEMS modeling, it seems that {\cb Theorem~\ref{MainThm} and the case of \eqref{HypBdry} in Theorem~\ref{BoundaryResult}}
may be of some importance in applications,
at least from the qualitative point of view (see \cite{YG-ZP-MJW06, EGG} for more details).
This is especially true for particular devices of MEMS type such as micro-valves, where the touchdown behavior is explicitly exploited,
since the touchdown or quenching set then corresponds to the lid or closing area of the valve.
As a consequence of our results, we see that  the latter has to be part of the positive set of the function $f$.
The choice of $f$, through an appropriate repartition of the dielectric coating, can thus be used in the optimal design of the microvalve.
In this respect, it would be desirable to gain further information about the structure of the quenching set,
but this seems a difficult mathematical problem for nonconstant $f$, even in one space dimension.
}

\medskip

{{\bf Remark.}} We point out that Theorems~\ref{MainThm}--\ref{BoundaryResult} still hold if we replace {\cb as in \cite{KMS08}}
the zero initial data by a nonnegative
$C^2$ function $u_0$ such that $u_0<1$ in $\overline\Omega$, $\Delta u_0+f(x)(1-u_0)^{-p}\ge 0$ in $\Omega$
and $u_0=0$ on $\partial\Omega$.
{Indeed, this assumption guarantees that $u_t>0$ and the proofs can then be modified
in a straightforward way.}


\section{Proof of Theorem 2}\label{point}
\setcounter{equation}{0}

\subsection{General strategy and basic computation}

When the compactness of the quenching set is known, type I estimates can
be proved by means of the maximum principle applied, {in a strict subdomain of $\Omega$,}
 to the well-known auxiliary function (cf. \cite{FM, G1}):
$$J(x,t):=u_t-\eps(1-u)^{-p},$$
where $\eps$ is a small positive constant.
In the present situation, the possible noncompactness of the quenching set prevents one to verify that
$J\ge 0$ on the boundary of any subdomain of $\Omega$
and the method is not directly applicable.

To overcome this, our basic idea is to consider a modified function $J$ as follows:
\be \label{defJah}
J{\cb (x,t)}=u_t-\eps a(x)h(u),
\ee
where $a(x)$ is an auxiliary function such that $a=0$ on $\partial\Omega$, hence also $J=0$.
The construction is delicate and requires specific properties for $a$, which will be given later.
As for the function $h(u)$, it will be a perturbation of the nonlinearity,
namely
\be \label{defhJ}
h(u)=(1-u)^{-p}+
(1-u)^{-q},\quad 0<q<p.
\ee
Before specializing, we first present the basic computations.

\begin{lemma} \label{lem1}
Let $J, h$ be given by (\ref{defJah})-(\ref{defhJ}), where $a\in C^2(\Omega)$ is a nonnegative function.
Then
\be\label{identJR}
J_t- \Delta J -pf(x)(1-u)^{-p-1}J = \eps R,
\ee
where
\be\label{defJR}
R=
(p-q)a(x)f(x)(1-u)^{-p-q-1}
+ah''(u)|\nabla u|^2
+2h'(u)\nabla a\cdot\nabla u+h(u)\Delta a.
\ee
Moreover, {$h''>0$ and,} at any point $x\in\Omega$ such that $a(x)>0$, we have
\be\label{estimJR}
R\ge
\underbrace{(p-q)a(x)f(x)(1-u)^{-p-q-1}}_{{\mathcal{T}}_1}
+\underbrace{h(u)\Delta a}_{{\mathcal{T}}_2}
-\underbrace{\frac{{h'}^2(u)|\nabla a|^2}{ah''(u)}}_{{\mathcal{T}}_3}.
\ee
\end{lemma}

\begin{proof}
We compute
\beaa
J_t&=&u_{tt}-\eps a(x)h'(u) u_t \\
\nabla J&=&\nabla u_t-\eps \bigl(a(x)h'(u)\nabla u+h(u)\nabla a(x)\bigr)\\
\Delta J &=&\Delta u_t-\eps \bigl(a(x)h'(u)\Delta u+a(x)h''(u)|\nabla u|^2+2h'(u)\nabla a(x)\cdot\nabla u+h(u)\Delta a(x)\bigr).
\eeaa
Setting $g(u)=(1-u)^{-p}$ and omitting the variables $x, u$ without risk of confusion, we get
\beaa
J_t-\Delta J
&=&(u_t-\Delta u)_t-\eps ah' (u_t-\Delta u)+\eps (ah'' |\nabla u|^2+2h'\nabla a\cdot\nabla u+h \Delta a), \\
\noalign{\vskip 1mm}
&=&f(x) g' u_t-\eps f(x) ah' g+\eps (ah'' |\nabla u|^2+2h'\nabla a\cdot\nabla u+h \Delta a).
\eeaa
Using $u_t=J+\eps ah$, we have
$$J_t- \Delta J -f(x) g' J
= \eps R,$$
where
$$R=f(x)a (g' h-h' g)+ ah'' |\nabla u|^2+2h'\nabla a\cdot\nabla u+h \Delta a.$$
On the other hand, we have
$$
\begin{array}{lll}
g'h-h' g
&=p(1-u)^{-p-1}\bigl((1-u)^{-p}+
(1-u)^{-q}\bigr)\cr
\noalign{\vskip 1mm}
&\ \ \ -(1-u)^{-p}\bigl(p(1-u)^{-p-1}+
q(1-u)^{-q-1}\bigr) \cr
\noalign{\vskip 1mm}
&=
(p-q)(1-u)^{-p-q-1},
\end{array}
$$
hence (\ref{defJR}).
Finally, {since $h''>0$,} for all $x\in \Omega$ such that $a(x)>0$, we may write
$$R=
(p-q)af(x)(1-u)^{-p-q-1} + h\Delta a+ah''\Bigl[|\nabla u|^2+2\frac{h'\nabla a\cdot\nabla u}{ah''}\Bigr],$$
hence (\ref{estimJR}).
\end{proof}

\subsection{Construction of the function $a(x)$}
We see that, in order to guarantee $R\ge 0$,
{the (negative) term ${\mathcal{T}}_3$ on the RHS of (\ref{estimJR}) must be absorbed by
a positive contribution coming either:
\smallskip

$\bullet$ from the term ${\mathcal{T}}_1$ (generated by the perturbation in (\ref{defhJ})), provided $f(x)>0$;  or
\smallskip

$\bullet$ from the term ${\mathcal{T}}_2$, provided $\Delta a(x)>0$.}
\smallskip

Since $a(x)$ is nonnegative and vanishes at the boundary, we cannot have $\Delta a>0$ everywhere.
Actually, we shall consider functions $a(x)$ which are positive in $\Omega$ and suitably convex everywhere except on a small ball $B$,
where $f$ is uniformly positive.
Also, it will be necessary to split the parabolic cylinder
$\Omega\times (T/2,T)$ into suitable sub-regions,
taking into account the ``large'' and ``small'' parts of
the function $u(x,t)$.

\medskip
The following essential lemma gives the construction of the appropriate function~$a(x)$.

\begin{lemma} \label{lem2}
Let $h$ be given by (\ref{defhJ}).
Let $x_0\in\Omega, \rho>0$ with $\overline B(x_0,\rho)\subset \Omega$,
and denote the open set
$$\Omega_{x_0,\rho}=\Omega\setminus \overline B(x_0,\rho).$$
Then there exists a function $a\in C^1(\overline\Omega)\cap C^2(\Omega)$ with the following properties:
\be\label{prop-ah1}
hh''a\Delta a-{h'}^2|\nabla a|^2\ge 0  \qh{for all $x\in\overline{\Omega_{x_0,\rho}}$ and all $0\le u<1$,}
\ee
\be\label{prop-ah2}
C_1\delta^{p+1}(x)\le a(x)\le C_{2}\delta^{p+1}(x) \qh{ for all $x\in\overline\Omega$},
\ee
for some constants $C_1, C_2>0$.
\end{lemma}

\begin{proof}
For $0<u<1$, computing
\bea
h'(u)&=&p(1-u)^{-p-1}+
q(1-u)^{-q-1}, \label{comput-hprime1} \\
\noalign{\vskip 1mm}
h''(u)&=&p(p+1)(1-u)^{-p-2}+
q(q+1)(1-u)^{-q-2},  \label{comput-hprime2}
\eea
it follows that
\beaa
hh''
&=&p(p+1)(1-u)^{-2p-2}+
(p(p+1)+q(q+1))(1-u)^{-p-q-2}\cr
\noalign{\vskip 2mm}
&&\ +
q(q+1)(1-u)^{-2q-2}\cr
\noalign{\vskip 2mm}
&\ge&
\frac{p+1}{p}\Bigl[p^2(1-u)^{-2p-2}+2
pq(1-u)^{-p-q-2}+
q^2(1-u)^{-2q-2}\Bigr],
\eeaa
{where we used} $p(p+1)+q(q+1)> 2(p+1)q$ due to $0<q<p$. Therefore, we have
\be\label{comp-h}
hh''\ge \frac{p+1}{p}(h')^2  \qh{ for $0\le u<1$.}
\ee

Next we introduce a suitable harmonic function $\phi$, namely, the unique solution of the problem
\beaa
&&\Delta \phi=0,\quad x\in\Omega_{x_0,\rho}, \\ 
&&\phi=0, \quad x\in \partial\Omega, \\ 
&&\phi=1,\quad x\in \partial B(x_0,\rho). 
\eeaa
The function $\phi$ is smooth and, by the strong maximum principle and the Hopf Lemma,
we have $0<\phi<1$ in $\Omega_{x_0,\rho}$ {\cb and
\be\label{prop-ah2b}
c_1\delta(x)\le \phi(x)\le c_2\delta(x),\quad x\in \overline{\Omega_{x_0,\rho}}
\ee
for some positive constants $c_1,c_2$.}
Then we set
\be\label{defaphi}
a(x)=\phi^{p+1}(x),\quad x\in\overline{\Omega_{x_0,\rho}}.
\ee
Since $a\in C^2(\Omega_{x_0,\rho})$, the boundary $\partial B(x_0,\rho)$ is smooth and $a=1$ on $\partial B(x_0,\rho)$,
the function $a$ can be extended in $B(x_0,\rho)$
in such a way that $a\in C^2(\Omega)$ and $a>0$ in $\Omega$.

On the other hand, on $\Omega_{x_0,\rho}$, we compute:
$$\nabla a=(p+1)\phi^p\nabla\phi,\qquad
\Delta a=(p+1)[\phi^p\Delta\phi+p\phi^{p-1}|\nabla\phi|^2]=p(p+1)\phi^{p-1}|\nabla\phi|^2,$$
{\cb hence}
\be\label{comp-a}
a\Delta a=\frac{p}{p+1} |\nabla a|^2   \qh{on $\Omega_{x_0,\rho}$.}
\ee
Combining (\ref{comp-h}) and (\ref{comp-a}), we get (\ref{prop-ah1}).
Property (\ref{prop-ah2}) follows from (\ref{prop-ah2b}), (\ref{defaphi})
{and $a>0$ in $\Omega$.}
\end{proof}

Before going further, let us recall the following useful lower bound on $u_t$.

\begin{lemma} \label{utHopf}
There exists a constant $c_0>0$ such that
\be\label{lowerut}
u_t\ge c_0 \delta(x) \qh{ on $\Omega\times [T/2,T)$.}
\ee
\end{lemma}

\begin{proof}
{\cb Although the proof of the lemma is standard (cf. \cite{FM,G1,G08}), we provide a proof here for completeness.}
Setting $v=u_t$, we see that $v$ satisfies
\beaa
&&v_t={\Delta u}+pf(x)(1-u)^{-p-1}v,\quad x\in\Omega,\; 0<t<T,\\   
&&v=0,\quad x\in\partial\Omega,\ 0<t<T,\\
&&v(x,0)=f(x),\quad x\in\Omega,
\eeaa
so that $v=u_t\ge 0$  in $Q_T:=\Omega\times(0,T)$ by the maximum principle.

Applying the maximum principle again, we deduce that $u_t\ge z$ in $Q_T$,
where $z$ is the solution of the heat equation in $Q_T$, with zero boundary condition and
initial condition $z(\cdot,0)=f$. Since $z$ satisfies the estimate~(\ref{lowerut})
in virtue of the Hopf lemma and the strong maximum principle, so does~$u_t$.
\end{proof}

\subsection{Proof of Theorem 2}

$ $

\medskip

{\bf Step 1.} {\it Preparations.}
Since $f\ge 0$ and $f\not\equiv 0$ is continuous,
we may pick a point $x_0\in \Omega$ and
$\rho>0$ such that $B(x_0,2\rho)\subset\Omega$ and
\be\label{inffx0}
\sigma_1:=
\inf_{x\in \overline B(x_0,\rho)}f(x)>0.
\ee
We then take $a\in C^1(\overline\Omega)\cap C^2(\Omega)$ as given by Lemma~\ref{lem2},
and define $J$ by
$$J(x,t)=u_t-\eps a(x)h(u),$$
with $\eps>0$ to be fixed later and
\be\label{defh2}
h(u)=(1-u)^{-p}+(1-u)^{-q}, \quad \hbox{\cb $0\le u<1$,\quad where $q=p/2$}
\ee
{\cb (any choice of $q\in (0,p)$ would do).}
Note that
\be\label{infax0}
\sigma_2:=
\inf_{x\in \overline B(x_0,\rho)}a(x)>0.
\ee
Next, we split the cylinder $\Sigma:=\Omega\times (T/2,T)$ into three
sub-regions as follows:
$$\Sigma_1=\bigl(\Omega\setminus \overline B(x_0,\rho)\bigr)\times (T/2,T),$$
$$\Sigma_2^\eta=\bigl\{(x,t)\in  \overline B(x_0,\rho)\times(T/2,T);\ u(x,t)\ge 1-\eta \bigr\}$$
and
$$\Sigma_3^\eta=\bigl\{(x,t)\in  \overline B(x_0,\rho)\times(T/2,T);\ u(x,t)< 1-\eta \bigr\},$$
where the number $\eta\in (0,1)$ will be specified later on.

\medskip

{\bf Step 2.} {\it Parabolic inequality for $J$ in the sub-regions
$\Sigma_1$ and $\Sigma_2^\eta$.}
It follows from properties~(\ref{estimJR}) in Lemma~\ref{lem1} and (\ref{prop-ah1}) in Lemma~\ref{lem2}, along with $a>0$, $f\ge 0$ {\cb in $\Omega$ and} $h''>0$, that
\be\label{ineqJSigma1}
J_t- \Delta J -pf(x)(1-u)^{-p-1}J\ge 0 \qh{ in $\Sigma_1$.}
\ee
Next, in view of {\cb \eqref{defh2} and} (\ref{comput-hprime1}), 
we have
\beaa
|h\Delta a|\le  {\cb C_3}(1-u)^{-p},\quad |h'\nabla a|\le {\cb C_3}(1-u)^{-p-1}
\qh{{in $\Sigma$}}
\eeaa
{\cb for some positive constant $C_3$ independent of $\varepsilon,\eta$. 
Also, from (\ref{comput-hprime2}) and \eqref{infax0} we get
\beaa
a h''\ge  \sigma_2 p(p+1)(1-u)^{-p-2} \qh{ in $\overline{B}(x_0,\rho)\times(0,T)$.}
\eeaa
}
Consequently, recalling the definition of $R$ in Lemma~\ref{lem1},
it follows from (\ref{estimJR}), (\ref{inffx0}) and (\ref{infax0}) that
\beaa
(1-u)^{p+q+1}R
&\ge& (p-q)f(x)a +h\Delta a(1-u)^{p+q+1}-\frac{(h'|\nabla a|)^2}{ah''}(1-u)^{p+q+1}\cr
\noalign{\vskip 1mm}
&\ge& (p-q)\sigma_1\sigma_2
-{\cb C_4} (1-u)^{q+1}
\ge (p-q)\sigma_1\sigma_2
-{\cb C_4}\eta^{q+1}  \qh{ in $\Sigma_2^\eta$,}
\eeaa
{\cb for some positive constant $C_4$ independent of $\varepsilon,\eta$.}
Owing to (\ref{identJR}), we may thus choose $\eta\in (0,1)$ small,
independent of $\eps$, such that
\be\label{ineqJSigma2}
J_t- \Delta J -pf(x)(1-u)^{-p-1}J\ge 0
 \qh{ in $\Sigma_2^\eta$.}
 \ee

\medskip

{\bf Step 3.}  {\it Control of $J$ on $\Sigma_3^\eta$ and conclusion.}
Now that $\eta$ has been fixed, using (\ref{prop-ah2}),
(\ref{lowerut}) and {\cb \eqref{defh2}}, 
we may choose $\eps>0$ small enough, so that
\be\label{ineqJSigma3}
J\ge \delta(x)\Bigl[c_0 - 2C_2\eps\delta^p(x) (1-u)^{-p}\Bigr]
\ge \delta(x)\Bigl[c_0 - 2C_2\eps\delta^p(x) \eta^{-p}\Bigr]\ge 0
  \qh{in $\Sigma_3^\eta$}
   \ee
  and
\be\label{ineqJinit}
J(x,T/2)\ge \delta(x)\Bigl[c_0 -2
C_2\eps\delta^p(x) \bigl(1-\|u(\cdot,T/2)\|_\infty\bigr)^{-p}\Bigr]\ge 0
  \qh{in $\overline\Omega$,}
\ee
{\cb where $c_0$ is the constant in Lemma~\ref{utHopf} and $C_1,C_2$ are the constants in \eqref{prop-ah2}.}
   Observe that, as a consequence of (\ref{ineqJSigma3})
   and $\Sigma=\Sigma_1\, \cup \,\Sigma_2^\eta\,\cup\, \Sigma_3^\eta$,
  we have
 \be\label{Csq-ineqJSigma}
 \bigl\{(x,t)\in \Sigma;\ J(x,t)<0\bigr\}\subset \Sigma_1\cup \Sigma_2^\eta.
 \ee
Also, since $a=0$ on $\partial\Omega$, we have
\be\label{ineqJboundary}
J=0 \qh{ on $\partial\Omega\times (T/2,T)$.}
\ee
On the other hand, by standard parabolic regularity,
we observe that
$$J\in C^{2,1}(\Sigma)\cap C(\overline\Omega\times [T/2,T)).$$

It follows from (\ref{ineqJSigma1})-(\ref{ineqJSigma2}),
(\ref{ineqJinit})-(\ref{ineqJboundary}) and
the maximum principle
(see e.g. \cite[Proposition 52.4 and Remark 52.11(a)]{QS})
that
$$J\ge 0 \qh{  in $\Sigma$.}$$
Then, for $T/2<t<s<T$ and $x\in \Omega$, taking (\ref{prop-ah2}) into account, an integration in time gives
\beaa
(1-u(x,t))^{p+1}
&\ge& (1-u(x,s))^{p+1}+C_1\eps(p+1)\delta^{p+1}(x)(s-t) \\
&\ge& C_1\eps(p+1)\delta^{p+1}(x)(s-t).
\eeaa
Letting $s\to T$, we finally deduce (\ref{rate})
in $\Sigma$, hence in $\Omega\times (0,T)$.
{\cb Note that the constants $C_1,\varepsilon$ are independent of $x,t$, so is the constant $\gamma$ in \eqref{rate}.} 
\qed

\section{Proof of Theorem 1}
\setcounter{equation}{0}

With the type~I estimate \eqref{rate} of Theorem~2 at hand, the proof is done via a suitable local comparison function. 
Let $x_0\in \Omega$ be such that $f(x_0)=0$ and take $b_0\in(0,1)$ small such that
\be\label{inclx0}
\overline B(x_0,2b_0)\subset \Omega.
\ee
We consider the following function
\beaa
w(x,t):=1-A\bigl[\phi(x)+(T-t)\bigr]^{1/(p+1)}
\qh{ in $\overline B(x_0,b)\times[0,T)$},
\eeaa
where
$$ \phi(x):=\kappa b^2 \left(1-\frac{|x-x_0|^2}{b^2}\right)^2.$$
Here, $\kappa\in (0,1)$ and $b\in (0,b_0)$ are constants to be chosen later and $A$ is a fixed positive constant
such that $A\le\gamma b_0$ and $A\le (1+T)^{-1/(p+1)}$ (where $\gamma$ is the constant given in~\eqref{rate}).
Note that
$$w(x,0)=1-A[\phi(x)+T]^{1/(p+1)}\ge 0 \qh{ for $x\in \overline B(x_0,b)$}$$
 and
\beaa
w(x,t)&=&1-A(T-t)^{1/(p+1)}\\
&\ge& 1-\gamma\delta(x)(T-t)^{1/(p+1)}\ge u(x,t)  \qh{ for $(x,t)\in\partial B(x_0,b)\times(0,T)$,}
\eeaa
due to (\ref{inclx0}), $A\le\gamma b_0$ and  \eqref{rate}.

We compute, in $B(x_0,b )\times(0,T)$,
\beaa
&&w_t- \Delta w-f(x)(1-w)^{-p}\\
&=&\frac{A}{p+1}\bigl[\phi(x)+(T-t)\bigr]^{-1+\frac{1}{p+1}} +\frac{A}{p+1}\bigl[\phi(x)+(T-t)\bigr]^{-1+\frac{1}{p+1}} \Delta\phi \\
&&\quad  -\frac{Ap}{(p+1)^2}\bigl[\phi(x)+(T-t)\bigr]^{-2+\frac{1}{p+1}} |\nabla \phi|^2-f(x)A^{-p}\bigl[\phi(x)+(T-t)\bigr]^{-p/(p+1)}\\
&=&\frac{A}{p+1}\bigl[\phi(x)+(T-t)\bigr]^{-p/(p+1)}\\
&&\quad \times \left\{1+\Delta\phi-\frac{p}{p+1}\bigl[\phi(x)+(T-t)\bigr]^{-1} |\nabla \phi|^2-(p+1)A^{-p-1} f(x) \right\},\\
\eeaa
hence
\bea
&&w_t- \Delta w-f(x)(1-w)^{-p}\label{Parabw}\\
&\ge&\frac{A}{p+1}\bigl[\phi(x)+(T-t)\bigr]^{-p/(p+1)}
\left\{1 +\Delta\phi  - \frac{p}{p+1}\frac{|\nabla \phi|^2}{\phi}-(p+1)A^{-p-1} f(x)\right\}. \notag
\eea
Moreover, in $B(x_0,b)$ we have
\beaa
&& \nabla\phi(x)=-4\kappa \left(1-\frac{|x-x_0|^2}{b^2}\right)\,(x-x_0),\\
&& \Delta\phi(x)=-4{n}\kappa \left(1-\frac{|x-x_0|^2}{b^2}\right)+8\kappa \left(\frac{|x-x_0|}{b}\right)^2
\ge -4{n}\kappa \\
&&\frac{|\nabla\phi(x)|^2}{\phi (x)} =16\kappa \left(\frac{|x-x_0|}{b}\right)^2 \le 16\kappa.  \\
\eeaa
Now, since $f(x_0)=0$, we may choose $b>0$ small enough so that
$$(p+1)A^{-p-1}\sup_{x\in B(x_0,b)}f(x)\le \frac{1}{3}.$$
Then we can choose $\kappa {=\kappa(n)}\in(0,1)$ small enough such that
\beaa
\Delta\phi(x)\ge -\frac{1}{3}, \quad \frac{|\nabla\phi(x)|^2}{\phi (x)}\le\frac{1}{3}\quad\mbox{ for all $x\in B(x_0,b)$.}
\eeaa
Therefore, $w_t- \Delta w-f(x)(1-w)^{-p}\ge 0$ in  $B(x_0,b )\times(0,T)$,
and it follows from the comparison principle that $u\le w$ in $B(x_0,b)\times(0,T)$.
{Since $\min_{\overline B(x_0,b/2)}\phi>0$,} this implies that $x=x_0$ is not a quenching point and the theorem is proved.
\qed

\bigskip

\section{Proof of Theorem 3}

(i) {\bf Case of (\ref{HypBdry1}).} Assume without loss of generality that $0\neq x_0\in\partial\Omega$ with $B(0,|x_0|)\cap\Omega=\emptyset$.
This (exterior ball condition) is possible due to the assumption that $\partial\Omega\in C^{2+\nu}$.
We look for a supersolution of the form $z(x)=1-C(d-r)^{\cb \beta}$, $r=|x|$, with $\beta>1$, $C{>0}, d>|x_0|$
to be chosen.
For $0<r<d$, we have:
\beaa
z_r&=&\beta C(d-r)^{\beta-1} \\
z_{rr}&=&-\beta (\beta-1)C (d-r)^{\beta-2}.
\eeaa
Set $\omega:=\Omega\cap\{x\in\BR^n;\ |x_0|<|x|<d\}$.
Choosing $\beta=2/(p+1)>1$, so that $\beta-2=-\beta p$, we compute in $\omega$:
\beaa
-\Delta z-f(x)(1-z)^{-p}&=&\beta (\beta-1)C(d-|x|)^{\beta-2}-\frac{\beta C(n-1)(d-|x|)^{\beta-1}}{|x|}\\
&&\ -f(x)\bigl[C(d-|x|)^\beta\bigr]^{-p} \\
&=&C(d-|x|)^{-\beta p}\Big[\beta(\beta-1)-\frac{\beta (n-1)(d-|x|)}{|x|}-C^{-p-1}f(x)\Bigr].
\eeaa
Next taking $d=d(p,n,|x_0|)>|x_0|$ close to $|x_0|$ and $C=C(p,\|f\|_\infty)>0$ large, we then have,
in $\omega$:
$$-\Delta z-f(x)(1-z)^{-p}\ge
C(d-|x|)^{-\beta p}\Big[\beta(\beta-1)-\frac{\beta (n-1)(d-|x_0|)}{|x_0|}-C^{-p-1}\|f\|_\infty\Bigr]\ge 0.$$
By taking $d$ possibly closer to $|x_0|$, we have
$$z(x)\ge 1-C(d-|x_0|)^\beta\ge 0\qh{ in $\overline\omega$},$$
hence also $z(x)\ge u(x,t)=0$ for $x\in \partial\omega\cap\{|x|<d\}=\partial\omega\cap\partial\Omega$ and $t\in (0,T)$.
Since on the other hand $z(x)=1>u(x,t)$ for $x\in \partial\omega\cap\{|x|=d\}$ and $0<t<T$, we deduce from the comparison principle that
$z\ge u$ on $\omega\times (0,T)$. Therefore $x_0$ is not a quenching point.
\medskip

(ii) {\bf Case of (\ref{HypBdry}).}  The proof {relies on estimate \eqref{rate} and} is a modification of that of Theorem~1.
Let $\Omega_\eta=\{x\in\Omega;\ \delta(x)<\eta\}$. {\cb There exists $\eta_0>0$ such that
$\Omega_\eta$ is a smooth bounded domain for all $\eta\in (0,\eta_0)$, due to $\Omega$ being a smooth domain.} We have
$\partial\Omega_\eta=\partial\Omega\cup\Gamma_\eta$, where $\Gamma_\eta=\{x\in\Omega;\ \delta(x)=\eta\}$.

Let $\psi_\eta$ be the unique solution of the problem
\beaa
&&\Delta \psi_\eta=0,\quad x\in\Omega_\eta, \\
&&\psi_\eta=1, \quad x\in \partial\Omega, \\
&&\psi_\eta=0,\quad x\in \Gamma_\eta. %
\eeaa
The function $\psi_\eta$ is smooth and, by the strong maximum principle,
we have $0<\psi_\eta<1$ in $\Omega_\eta$.
Letting
$$\phi=k\psi_\eta^2,$$
we consider the function
\beaa
w(x,t):=1-\gamma\eta\bigl[\phi(x)+(T-t)\bigr]^{1/(p+1)}
\qh{ in $\overline\Omega_\eta\times [0,T)$},
\eeaa
where $k\in (0,1)$ and $\eta\in (0,\eta_0)$ are constants to be chosen later {\cb and $\gamma$ is the constant given in \eqref{rate}.}
First assuming $\eta\le \eta_1:=\min\bigl(\eta_0,\gamma^{-1}(1+T)^{-1/(p+1)}\bigr)$, we have
$$w(x,0)=1-\gamma\eta[\phi(x)+T]^{1/(p+1)}\ge 0 \qh{ in $\overline\Omega_\eta,$}$$
along with $w(x,t)\ge 0$ on $\partial\Omega\times (0,T)$ and, by \eqref{rate},
\beaa
w(x,t)&=&1-\gamma\eta(T-t)^{1/(p+1)}= 1-\gamma\delta(x)(T-t)^{1/(p+1)}\ge u(x,t)  \qh{ on $\Gamma_\eta\times(0,T)$.}
\eeaa

Formula (\ref{Parabw}) remains valid, with $A$ replaced by $\gamma\eta$,
and moreover we have
$$\Delta\phi=2k|\nabla\psi_\eta|^2 \qh{ and }\quad \frac{|\nabla\phi|^2}{\phi}=4k|\nabla\psi_\eta|^2.$$
Therefore,
\beaa
&&w_t-\Delta w-f(x)(1-w)^{-p}\\
&\ge&\frac{\gamma\eta}{p+1}\bigl[\phi(x)+(T-t)\bigr]^{-p/(p+1)}
\left\{1-\frac{4kp}{p+1} |\nabla \psi_\eta|^2-(p+1)(\gamma\eta)^{-p-1} f(x)\right\}
\eeaa
{in $\Omega_\eta\times [0,T)$.} Now, by assumption (\ref{HypBdry}), we may choose $\eta\in(0,\eta_1)$ small enough so that
$$(p+1)(\gamma\eta)^{-p-1}\sup_{x\in\Omega_\eta}f(x)\le \frac{1}{2}.$$
Then we can choose $k=k(\eta)>0$ small enough so that
\beaa
4k|\nabla\psi_\eta|^2\le\frac{1}{2}\quad\mbox{ for all $x\in\Omega_\eta$.}
\eeaa
Therefore, $w_t-\Delta w-f(x)(1-w)^{-p}\ge 0$ in $\Omega_\eta\times(0,T)$,
and it follows from the comparison principle that $u\le w$ in $\Omega_\eta\times(0,T)$.
Since $\min_{\overline\Omega_{\eta/2}}\psi_\eta>0$, this guarantees that no quenching occurs near the boundary.
\qed

\medskip

{{\bf Remark.} Although it is not clear if Theorem 3(i) has a direct relation with this fact,
it is interesting to recall that, when $f$ is constant, quenching for problem (\ref{pdeGen})-(\ref{icGen}) is
incomplete (in the sense of existence of a suitable weak continuation after $t=T$)
if and only if $0<p<1$ (cf.~\cite{Phi87, FLV, GV97}).}


\begin{thebibliography}{111}

\bibitem{EGG07}
P. Esposito, N. Ghoussoub and Y. Guo,
\textit{Compactness along the first branch of unstable solutions for an elliptic problem with a singular nonlinearity},
Comm. Pure Appl. Math. \textbf{60} (2007), 1731--1768.

{\bibitem{E08} P. Esposito, \textit{Compactness of a nonlinear eigenvalue problem with a singular nonlinearity},
Comm. Contemp. Math. \textbf{10} (2008), 17--45.}

\bibitem{EG08} P. Esposito and N. Ghoussoub, \textit{Uniqueness of solutions for an elliptic equation modeling MEMS},
Methods Appl. Anal. \textbf{15} (2008), 341--354.

\bibitem{EGG} P. Esposito, N. Ghoussoub and Y. Guo,
Mathematical analysis of partial differential equations modeling electrostatic MEMS. Courant Lecture Notes in Mathematics, 20. Courant Institute of Mathematical Sciences, New York; American Mathematical Society, Providence, RI, 2010. xiv+318 pp.

\bibitem{FLV}
M. Fila, H.A. Levine, and J.L. V\'azquez,
\textit{Stabilization of solutions of weakly singular quenching problems,}
Proc. Amer. Math. Soc. \textbf{119} (1993), 555--559.

{\bibitem{FMPS07} G. Flores, G. Mercado, J.A. Pelesko and N. Smyth,
\textit{Analysis of the dynamics and touchdown in a model of electrostatic MEMS},
SIAM J. Appl. Math. \textbf{67} (2007), 434--446.}

\bibitem{FM} A. Friedman, B. McLeod,
\textit{Blow-up of positive solutions of semilinear heat equations.}
Indiana Univ. Math. J. \textbf{34} (1985), 425--447.

\bibitem{GV97}
 V.A. Galaktionov, J.L. V\'azquez,
\textit{ Continuation of blowup solutions of nonlinear heat equations in several space dimensions,}
 Comm. Pure Appl. Math.  \textbf{50} (1997), 1--67.


\bibitem{GG07}
N. Ghoussoub and Y. Guo,
\textit{On the partial differential equations of electrostatic MEMS devices: stationary case},
SIAM J. Math. Anal. \textbf{38} (2007), 1423--1449.


{\bibitem{GG08} N. Ghoussoub and Y. Guo,
\textit{On the partial differential equations of electrostatic MEMS devices II: dynamic case},
Nonlinear Diff. Eqns. Appl. \textbf{15} (2008), 115--145.}

\bibitem{G1} J.-S. Guo,
{\it On the quenching behavior of the solution of a semilinear parabolic equation},
J. Math. Anal. Appl. {\bf 151} (1990), 58-79.


\bibitem{G97} J.-S. Guo,
\textit{Quenching problem in nonhomogeneous media},
Differential and Integral Equations \textbf{10} (1997), 1065-1074.

\bibitem{G-s} J.-S. Guo,
\textit{Recent developments on a nonlocal problem arising in the micro-electro mechanical system},
Tamkang J. Math. \textbf{45} (2014), 229--241.

\bibitem{GHW08} J.-S. Guo, B. Hu and C.-J. Wang,
\textit{A nonlocal quenching problem arising in micro-electro mechanical systems}, Quarterly Appl. Math.
\textbf{67} (2009), 725--734.

\bibitem{GK} J.-S. Guo and N.I. Kavallaris,
\textit{On a nonlocal parabolic problem arising in electrostatic MEMS control},
Discrete and Continuous Dynamical Systems \textbf{32} (2012), 1723--1746.


\bibitem{G08} Y. Guo, \textit{On the partial differential equations of electrostatic MEMS devices III: refined
touchdown behavior}, J. Diff. Eqns., \textbf{244} (2008), 2277--2309.

\bibitem{G08A} Y. Guo, \textit{Global solutions of singular parabolic equations arising from electrostatic
MEMS}, J. Diff. Eqns. \textbf{245} (2008), 809--844.

\bibitem{YG-ZP-MJW06}
Y. Guo, Z. Pan, and M.J. Ward,
\textit{Touchdown and pull-in voltage behavior of a MEMS device with varying dielectric properties},
SIAM J.Appl. Math {\textbf{66} (2005)}, 309--338.

{\bibitem{GW08} Z. Guo and J. Wei, \textit{Asymptotic behavior of touchdown solutions and global bifurcations
for an elliptic problem with a singular nonlinearity}, Comm. Pure Appl. Anal. \textbf{7} (2008), 765--786.}


\bibitem{KMS08} N.I. Kavallaris, T. Miyasita, and T. Suzuki,
\textit{Touchdown and related problems in electrostatic MEMS device equation}, Nonlinear Diff. Eqns. Appl.
\textbf{15} (2008), 363--385.

\bibitem{L89} H.A. Levine,
\textit{Quenching, nonquenching, and beyond quenching for solution of some parabolic
equations}, Ann. Mat. Pura Appl. \textbf{155} (1989), 243--260.

\bibitem{PT01}
J.A. Pelesko and A.A. Triolo, \textit{Nonlocal problems in MEMS device control,} J. Engrg. Math., \textbf{41} (2001), 345--366.

\bibitem{JAP-DHB02}
J.A. Pelesko and D.H. Bernstein,
\textit{Modeling MEMS and NEMS},
Chapman Hall and CRC Press, 2002.

\bibitem{Phi87}
D. Phillips,
\textit{Existence of solutions of quenching problems,}
Applicable Anal. 24 (1987), 253--264.


\bibitem{QS}
P. Quittner, Ph. Souplet,
Superlinear parabolic problems. Blow-up, global existence and steady states,
Birkh\"auser Advanced Texts, 2007.

\bibitem{YZ10}
D. Ye and F. Zhou,
\textit{On a general family of nonautonomous elliptic and parabolic equations},
Calc. Var. \textbf{37} (2010), 259--274.

\end{thebibliography}
\end{document}